\documentclass[oneside,11pt]{amsart}
\usepackage{amssymb,amscd,amsmath,latexsym}
\usepackage[mathcal]{euscript}
\usepackage[colorlinks]{hyperref}
\hypersetup{breaklinks=true,
pagecolor=white}
\newtheorem{thm}{Theorem}[section]

\newtheorem{lem}[thm]{Lemma}
\newtheorem{prop}[thm]{Proposition}
\theoremstyle{definition}
\newtheorem{defin}[thm]{Definition}
\newtheorem{exm}[thm]{Example}

\newtheorem{prob}[thm]{Problem}
 \numberwithin{equation}{section}
 \numberwithin{figure}{section}

\theoremstyle{remark}
\newtheorem{rmk}[thm]{Remark}

\newcommand{\Ext}{\operatorname{Ext}}

\newcommand{\Hom}{\operatorname{Hom}}

\newcommand{\ks}{k\Sigma_d}

\newcommand{\sgn}{\operatorname{sgn}}

\newcommand{\HH}{\operatorname{H}}

\begin{document}

\title[Specht module cohomology]
{\bf A combinatorial approach to Specht module cohomology.}
\author{\sc David J. Hemmer}
\address{Department of Mathematics\\ University at Buffalo, SUNY \\
244 Mathematics Building\\Buffalo, NY~14260, USA}
\thanks{Research of the  author was supported in part by NSF
grant  DMS-0808968} \email{dhemmer@math.buffalo.edu}

\date{October 2009}


\begin{abstract}
For a Specht module $S^\lambda$ for the symmetric group $\Sigma_d$, the cohomology $\HH^i(\Sigma_d, S^\lambda)$ is known only in degree $i=0$. We give a combinatorial criterion equivalent to the nonvanishing of the degree $i=1$  cohomology, valid in odd characteristic. Our condition generalizes James' solution in degree zero. We apply this combinatorial description to give some  computations of Specht module cohomology, together with an explicit description of the corresponding modules. Finally we suggest some general conjectures that might be particularly amenable to proof using this description.
\end{abstract}
\maketitle

\section{Introduction}
\label{sec: Introduction} For a finite group $G$ and a $G$-module $M$ defined over a field $k$, computing cohomology groups with coefficients in $M$,  $\HH^i(G,M)=\Ext^i_G(k,M)$, is often difficult. For the symmetric group even the $i=1$ case is unknown for important classes of modules like Specht modules or irreducible modules. For $M$ a Specht module  and $i=1$ we describe a straightforward combinatorial condition equivalent to the nonvanishing of this cohomology group. We hope this approach may be useful in resolving some  conjectures about these cohomology groups, discussed in Section \ref{sec: further directions}.

For a partition $\lambda$ of $d$, let $S^\lambda$ denote the corresponding Specht module for the symmetric group $\Sigma_d$ and let $S_\lambda$ be its dual. (For descriptions of these modules and general information on symmetric group representation theory see \cite{Jamesbook}.) In even characteristic every Specht module is a dual Specht module and the problem of computing cohomology seems more difficult. Further, our combinatorial description below is false in characteristic two, so we will mostly focus on the case of odd characteristic.

There has been some success understanding $H^i(\Sigma_d, S_\lambda)$ for small $i$. In \cite{BKMdualspecht} (where only odd characteristic is considered), it is shown that the cohomology vanishes in degrees $1 \leq i \leq p-3$. For $p=3$ a complete description was given for $i=1, 2$.

For Specht modules only $i=0$ is completely understood. In \cite[Theorem 24.4]{Jamesbook}, James computes the invariants
 $\HH^0(\Sigma_d, S^\lambda) =\Hom_{\ks}(k, S^\lambda).$
The module $S^\lambda$ is a submodule of the transitive permutation module $M^\lambda$, and $\Hom_{\ks}(k, M^\lambda)$
is one-dimensional. So to compute $\Hom_{\ks}(k, S^\lambda)$ one needs to know whether the one-dimensional fixed-point space in $M^\lambda$ lies in $S^\lambda$ or not. James' proof is essentially combinatorial, using the kernel intersection theorem (Corollary 17.18 in \cite{Jamesbook} and Theorem \ref{thm: Kernel intersection theorem} below) to test if the fixed point lies in $S^\lambda$. The solution involves determining when certain binomial coefficients are divisible by $p$. A similar theme arises below in Section \ref{sec: two more general examples}.

We generalize James' work as follows. First we prove that any nonsplit extension of $S^\lambda$ by the trivial module is isomorphic to a submodule of $M^\lambda$. Then, using the kernel intersection theorem, we prove a combinatorial condition on a vector $u \in M^\lambda$ that is equivalent to the subspace $\langle S^\lambda, u \rangle$ being the nonsplit extension we desire. Next we apply the result to do some computations. The novelty here is that we can compute cohomology in a purely combinatorial way, without understanding projective resolutions, and that we end up with an explicit basis for the corresponding nonsplit extension. Finally we suggest some general conjectures that may be attacked with this result, and indeed which formed the motivation for this paper.

\section{Semistandard Homomorphisms and the Kernel Intersection Theorem}
\label{Sec: Semistandardhomomorphisms}
In this section we describe the kernel intersection theorem. A \emph{composition} of $d$ is a sequence $(\lambda_1, \lambda_2, \ldots)$ of nonnegative integers that sum to $d$. If the $\lambda_i$ are nonincreasing we say $\lambda$ is a \emph{partition} of $d$, and write $\lambda \vdash d$.
Denote by $[\lambda]$ the \emph{Young diagram} for $\lambda$:
$$[\lambda]= \{(i,j) \in {\mathbb N} \times {\mathbb N} \mid j \leq \lambda_i\}. $$
A $\lambda$-\emph{tableau} is an assignment of  $\{1, 2, \ldots, d\}$ to the boxes in $[\lambda]$. The symmetric group acts transitively on the set of $\lambda$-tableau. For a tableau $t$ its row stabilizer $R_t$ is the subgroup of $\Sigma_d$ fixing the rows of $t$ setwise. Say $t$ and $s$ are equivalent if $t=\pi s$ for some $\pi \in R_s$. An equivalence class is called a $\lambda$-\emph{tabloid}, and the class of $t$ is denoted $\{t\}$. The vector space with the set of $\lambda$-tabloids as a basis is the permutation module $M^\lambda$. If $\lambda=(\lambda_1, \lambda_2, \ldots, \lambda_s)$, there is a corresponding \emph{Young subgroup}
    $$\Sigma_\lambda \cong \Sigma_{\lambda_1} \times \cdots \times \Sigma_{\lambda_s}\leq \Sigma_d.$$
    The stabilizer of a $\lambda$-tabloid $\{t\}$ is clearly a conjugate of $\Sigma_\lambda$ so we have:
    \begin{equation}
            \label{eq: DefofMlambdainduced}
        M^\lambda \cong \operatorname{Ind}_{\Sigma_\lambda}^{\Sigma_d} k.
    \end{equation}

Since $M^\lambda$ is a transitive permutation module, it has a one-dimensional fixed-point space under the action of $\Sigma_d$. Let $f_\lambda \in M^\lambda$ denote the sum of all the $\lambda$-tabloids, so $f_\lambda$ spans this fixed subspace.









The Specht module $S^\lambda$ is defined explicitly as the submodule of $M^\lambda$ spanned by certain linear combinations of tabloids, called polytabloids. In characteristic zero the Specht modules $\{S^\lambda \mid \lambda \vdash d\}$  are  a complete set of nonisomorphic simple $\Sigma_d$-modules. James gave an alternate description of $S^\lambda$ inside $M^\lambda$ as the intersection of the kernels of certain homomorphism from $M^\lambda$ to other permutation modules.

Let $\lambda=(\lambda_1, \lambda_2, \ldots) \vdash d$ and let $\nu=(\lambda_1, \lambda_2, \ldots, \lambda_{i-1}, \lambda_i+\lambda_{i+1}-v, v, \lambda_{i+2}, \ldots).$ James defined \cite[Definition 17.10]{Jamesbook} the module homomorphism $\psi_{i,v}: M^\lambda \rightarrow M^\nu$ by:

\begin{equation}
\label{eq: defofpsiiv}
    \psi_{i,v}(\{t\})=\sum \left\{ \{t_1\} \mid \begin{array}{l}  \{t_1\} \text{ agrees with } \{t\} \text { on all except row $i$ and $i+1$,}\\ \text{and row $i+1$ of } \{t_1\} \text { is a subset of size $v$ in row $i+1$ of } \{t\}. \end{array} \right\}
        \end{equation}
        Notice that every $\nu$- tabloid  in $\psi_{i,v}(\{t\})$ has coefficient at most one. James proved:

\begin{thm}[Kernel Intersection Theorem] \cite[17.18]{Jamesbook}
\label{thm: Kernel intersection theorem}
Suppose $\lambda \vdash d$ has $r$ nonzero parts. Then
$$S^\lambda= \bigcap_{i=2}^r \bigcap_{v=0}^{\lambda_i-1}\ker (\psi_{i-1,v}) \subseteq M^\lambda$$
\label{prop: JamesKernelInt}
\end{thm}

Given a linear combination of tabloids $u \in M^\lambda$,  Theorem \ref{thm: Kernel intersection theorem} gives an explicit test for whether $u \in S^\lambda$. If $u=f_\lambda$, then $u$ spans the one-dimensional fixed-point space in $M^\lambda$. Applying the test in this case let James determine when  $\HH^0(\Sigma_d, S^\lambda)$ is nonzero, as follows. For an integer $t$ let $l_p(t)$ be the least nonnegative integer
satisfying $t<p\,^{l_p(t)}$. James proved:

\begin{thm}\cite[24.4]{Jamesbook}
\label{thm:JamestheoremonHom}
$\HH^0(\Sigma_d, S^\lambda)$ is zero unless  $\lambda_i \equiv -1 \mbox{ \rm mod }
p\,^{l_p(\lambda_{i+1})}$ for all $i$, in which case it is one-dimensional.
\end{thm}

\section{Nonsplit extensions inside permutation modules.}
\label{sec: nonsplitextensionsinsideperm}
It is immediate that $\Hom_{\ks}(k, S^\lambda)$ is ``determined by" $M^\lambda$, since $S^\lambda \subseteq M^\lambda$. In this section we prove, in odd characteristic, that $\Ext^1_{\ks}(k, S^\lambda)$ is also determined completely by the structure of $M^\lambda$. First a lemma:

\begin{lem}Let $\lambda \vdash d.$
\label{lem: basicpropsofMlambda}
\begin{itemize}
  \item [(a)] $\Hom_{\ks}(k, M^\lambda) \cong k$.
  \item [(b)] For $p>2$, $\Hom_{\ks}(S^\lambda, M^\lambda) \cong k.$
  \item [(c)] For $p>2$, $\Ext^1_{\ks}(k, M^\lambda) =0$.
\end{itemize}
\end{lem}
\begin{proof}
Parts (a) and (c) follow from \eqref{eq: DefofMlambdainduced} and the Eckmann-Shapiro lemma, since $\Ext^1_{\ks}(k,k)=0$ in odd characteristic. Part (b) is  \cite[Corollary 13.17 ]{Jamesbook}.
\end{proof}

The next result says that, for $p>2$, any nonsplit extension of $S^\lambda$ by the trivial module embeds in $M^\lambda$.
\begin{thm}
\label{thm: anynonsplitsitsinsideMlambda}
Suppose $p >2$ and suppose there is a nonsplit short exact sequence:
\begin{equation}
\label{eq: SESnonsplit}
0 \rightarrow S^\lambda \rightarrow U \rightarrow k \rightarrow 0.
 \end{equation}
 Then $U$ is isomorphic to a submodule of $M^\lambda$.
\end{thm}
\begin{proof}
Apply $\Hom_{\ks}(-, M^\lambda)$ to \eqref{eq: SESnonsplit} to obtain a long exact sequence. From Lemma \ref{lem: basicpropsofMlambda} we get:
    \begin{equation}
            \label{eq: longexactsequenceembedsinmlambda}
                0 \rightarrow k \rightarrow \Hom_{\ks}(U, M^\lambda) \stackrel{f}{\rightarrow} \Hom_{\ks}(S^\lambda, M^\lambda)\cong k \rightarrow 0.
    \end{equation}
Thus the map $f$ in \eqref{eq: longexactsequenceembedsinmlambda} is surjective, so the embedding of $S^\lambda$ into $M^\lambda$ lifts to a map  $\tilde{f}: U \rightarrow M^\lambda$ which is faithful on $S^\lambda$.  Thus $\tilde{f}$ must also be injective, otherwise the sequence \eqref{eq: SESnonsplit} would be split.
\end{proof}

\begin{rmk}
Theorem \ref{thm: anynonsplitsitsinsideMlambda} is false in characteristic 2, indeed the case $\lambda=(2) \vdash 2$ is a counterexample. The Specht module $S^{(2)}$ is trivial and there is a nonsplit extension of $k$ by $k$ which clearly is not a submodule of the one-dimensional module $M^{(2)}$.
\end{rmk}

\begin{rmk}
Lemma \ref{lem: basicpropsofMlambda}(b) and (c) hold with $M^\lambda$ replaced by the Young module $Y^\lambda$ (which is the unique indecomposable direct summand of $M^\lambda$ containing the Specht module $S^\lambda$). Thus Theorem 3.2 could be ``strengthened" to say that $U$ is isomorphic to a submodule of  $Y^\lambda$. At present we have no way to use this since not even the dimension of $Y^\lambda$ is known, let alone a combinatorial description of it as a submodule of $M^\lambda$.
\end{rmk}

\begin{rmk}If $p>3$ then $\Ext^1_{\ks}(\sgn, k)=0$, and hence $\Ext^1_{\ks}(\sgn, M^\lambda)=0$. An argument as in Theorem \ref{thm: anynonsplitsitsinsideMlambda} would imply any nonsplit extension of $S^\lambda$ by the sign module appears as a submodule of $M^\lambda$. However there are no such extensions as
$$\Ext^1_{\ks}(\sgn, S^\lambda) \cong \HH^1(\Sigma_d, S_{\lambda'})=0$$ by \cite{BKMdualspecht}. More generally one could ask if a nonsplit extension of $S^\lambda$ by an irreducible module must embed in $M^\lambda$. In the case where $\lambda$ is $p$-restricted the answer is always yes, as the Young module $Y^\lambda$ is injective. The general problem seems to remain open.
\end{rmk} 

Theorem \ref{thm: anynonsplitsitsinsideMlambda} says that $M^\lambda$ completely controls $\HH^1(\Sigma_d, S^\lambda)$. In particular, $\HH^1(\Sigma_d, S^\lambda)$ is nonzero precisely when such a $U$ exists inside $M^\lambda$. Constructing such a $U$ is equivalent to finding a vector $u \not \in S^\lambda$ such that subspace $\langle S^\lambda, u \rangle$ gives the desired module. Necessary and sufficient conditions on such a $u$ are given next.

\begin{thm}
\label{thm: requirementonu}
Let $p > 2$ and $\lambda \vdash d$. Then $\Ext^1_{\ks}(k,S^\lambda) \neq 0$ if and only if there exists $u \in M^\lambda$ with the following properties:
\begin{enumerate}
\item For each $\psi_{i,v}:  M^\lambda \rightarrow M^\nu$  appearing in Theorem  \ref{prop: JamesKernelInt},  $\psi_{i,v}(u)$ is a multiple of $f_\nu$,  at least one of which is a nonzero multiple.
\item There does not exist an $a \neq 0$ such that all the $\psi_{i,v}(af_\lambda - u)$ are zero.
\end{enumerate}
If so then the subspace spanned by $S^\lambda$ and u is a submodule that is a nonsplit extension of $S^\lambda$ by $k$.
\end{thm}

\begin{proof}
Suppose $\Ext^1_{\ks}(k,S^\lambda) \neq 0$. From Theorem \ref{thm: anynonsplitsitsinsideMlambda} we can find a submodule $U$ of $M^\lambda$ that is a nonsplit extension of $S^\lambda$ by $k$. Choose $u \in M^\lambda$ such that $U$ is spanned by $u$ and $S^\lambda$. Since $S^\lambda$ is in the kernel of all the $\psi_{i,v}$'s, and $U/S^\lambda \cong k$, we have for $\sigma \in \Sigma_d$ that $\sigma(u)=u+v$ for some $v \in S^\lambda$. Applying $\psi_{i,v}$ we get that $$\sigma(\psi_{i,v}(u))=\psi_{i,v}(u)\,\, \forall \sigma \in \Sigma_d,$$ so $\psi_{i,v}(u)$ is a multiple of $f_\nu$. But $u \not\in S^\lambda$ so not all the multiples are zero, and thus $(1)$ above holds for $u$. To see $(2)$ holds, suppose there is an $a \neq 0$ that $\psi_{i,v}(af_\lambda - u) =0$ for all $\psi_{i,v}$. This means $af_\lambda-u \in S^\lambda$ by  the kernel intersection theorem. If $af_\lambda \in S^\lambda$ then so is $u$, a contradiction. Otherwise $f_\lambda \in U$ but $f_\lambda \not\in S^\lambda$ and so $U \cong S^\lambda \oplus k$, a contradiction to $U$ being a nonsplit extension.

Conversely suppose such a $u$ exists. For any $\sigma \in \Sigma_d$, notice that:
    \begin{eqnarray*}
          \psi_{i,v}(\sigma(u)-u) &=& \sigma \psi_{i,v}(u)-\psi_{i,v}(u) \\
         &=& 0 \textrm{ by } (1).
    \end{eqnarray*}
Thus $\sigma(u) -u \in S^\lambda$, so $U:=\langle S^\lambda, u \rangle $ is a submodule of $M^\lambda$ such that $U/S^\lambda \cong k$. Condition (2) ensures $U$ is not a direct sum of $S^\lambda$ and the one-dimensional trivial submodule of $M^\lambda$.

\end{proof}

\begin{rmk}
\label{rmk: noneedforsecondcondition}
When $f_\lambda \not\in S^\lambda$ then $U=S^\lambda \oplus \langle f_\lambda \rangle$ is a submodule such that $U/S^\lambda \cong k$, which is a split extension of $S^\lambda$ by $k$. The condition on $u$ given by (2) in the theorem ensures that the submodule $\langle u, S^\lambda \rangle$ is not this split extension. In particular condition (2) is implied by condition (1) when $f_\lambda \in S^\lambda$, i.e. for the $\lambda$ where $\Hom_{\ks}(k,S^\lambda) \neq 0$. These $\lambda$ are given by Theorem \ref{thm:JamestheoremonHom}.
\end{rmk}

\begin{rmk}
\label{rmk: manychoiceofu}
The choice of $u$ in the theorem is far from unique. Given a $u$ that works any other vector of the form $u +v$ for $v \in S^\lambda$ will also work. Thus there is some strategy involved in the choice of $u$ for proving theorems.
\end{rmk}

\section{Two small examples}
\label{sec: Some examples}

In this section we give two small examples to illustrate Theorem \ref{thm: requirementonu}. In the next section we obtain general results that generalize these examples. For a two-part partition $\lambda=(\lambda_1, \lambda_2)$, a $\lambda$-tabloid is entirely determined by the entries of its second row, and we will represent them by just the last row, with a bar over to reflect the equivalence relation. For example the tabloid:
$$\{t\}=\left\{\begin{array}{ccc}
    1 & 5 & 2 \\
    3 & 4 &
  \end{array} \right\}$$
  will be denoted $\overline{34}$. For $\lambda=(d)$ the unique $\lambda$-tabloid will be denoted $\overline{\emptyset}$.
\begin{exm}
\label{exm :33p=3} Let $p=3$ and $\lambda=(3,3) \vdash 6.$ Define $u \in M^{(3,3)}$ by:
\begin{eqnarray}
\label{eq: ufor33}
u&=&\overline{134}+\overline{135}+\overline{136}+
\overline{145}+\overline{146}+\overline{156}+\overline{234}+\overline{235}+\overline{236}
+\overline{245}+\overline{246}+\overline{256}\\&&-\overline{123}-\overline{124}
-\overline{125}-\overline{126}.\nonumber
\end{eqnarray}
One easily checks that:
\begin{eqnarray*}
\psi_{1,0}(u)&=&(12-4)\overline{\emptyset} \equiv -f_{(6)}\\
\psi_{1,1}(u)&=&(6-4)(\overline{1}+\overline{2})+(6-1)(\overline{3}+\overline{4}+\overline{5}+\overline{6}) \equiv -f_{(5,1)}\\
\psi_{1,2}(u)&=&-4(\overline{12})+(3-1)(\overline{13}+\cdots +\overline{56}) \equiv -f_{(4,2)}.
\end{eqnarray*}
Thus $u$ satisfies condition $(1)$ of Theorem \ref{thm: requirementonu}. Since $f_{(3,3)} \not \in S^{(3,3)}$ we must also check condition (2). In fact:
$$\psi_{1,1}(af_{(3,3)}-u) =(-a-1)f_{(6)}, \,\,\,\, \psi_{1,0}(af_{(3,3)}-u)=(a-1)f_{(5,1)}.$$ There is no such $a \neq 0$ which makes both these zero. Thus condition (2) holds and we conclude:
\begin{prop}
In characteristic three, $\HH^1(\Sigma_6, S^{(3,3)}) \neq 0$. Further, the subspace of $M^{(3,3)}$ spanned by $S^{(3,3)}$ and $u$ from \eqref{eq: ufor33} is a submodule that is a nonsplit extension of $S^{(3,3)}$ by $k$.
\end{prop}

\end{exm}

\begin{exm}
\label{exm: 83}
Next consider $\lambda=(8,3)$ in characteristic 3.  Let
\begin{equation}
\label{eq: ufor83}
u=\sum \left\{\{t\} \in M^{(8,3)} \mid 1,2,\text{ and } 3 \text{ appear in the first row of } \{t\} \right\}.
\end{equation}

One easily checks that:
\begin{eqnarray*}\psi_{1,0}(u) &=& \binom{8}{3}f_{(6)} \equiv -f_{(6)}\\
\psi_{1,1}(u)&=&\binom{7}{2}(\overline{4}+ \cdots + \overline{11}) \equiv 0\\
\psi_{1,2}(u)&=&\binom{6}{1}(\overline{45} + \overline{46}+ \cdots + \overline{10 \, 11}) \equiv 0 \end{eqnarray*}
so $u$ satisfies condition (1) in Theorem \ref{thm: requirementonu}. Since $f_{(8,3)} \in S^{(8,3)}$ by Theorem \ref{thm:JamestheoremonHom}, condition $(2)$ does not need to be checked (see Remark \ref{rmk: noneedforsecondcondition}).

\end{exm}

Thus we have:

\begin{prop}
In characteristic three, $\HH^1(\Sigma_{11}, S^{(8,3)}) \neq 0$. Further, the subspace of $M^{(8,3)}$ spanned by $S^{(8,3)}$ and $u$ from \eqref{eq: ufor83} is a submodule that is a nonsplit extension of $S^{(8,3)}$ by $k$.
\end{prop}
Example \ref{exm: 83} has the same flavor as the proof of Theorem  \ref{thm:JamestheoremonHom} in \cite[p.101-102]{Jamesbook}, in that knowing the congruence class of binomial coefficients modulo $p$ is important, specifically knowing that $\binom{7}{2}$ and $\binom{6}{1}$ are both zero modulo three. The two examples in the next section are more general results, generalizing the previous two examples and illustrating this theme.

\section{Two more general examples.}
\label{sec: two more general examples}
The examples in this section are already known, in the sense that  $\HH^1(\Sigma_d, S^{(\lambda_1, \lambda_2)})$ is known in odd characteristic by work in \cite{HNcohomologyspechtmodules}. However the proof there uses work of Erdmann on the special linear group $SL_2(k)$ and Schur functor techniques. The proofs here are purely combinatorial and completely contained within the symmetric group theory. Moreover, the extensions constructed with Theorem \ref{thm: requirementonu} come equipped with an explicit basis and a description of the $\Sigma_d$-action, which is new even for these cases.

At several times we will need to know when a binomial coefficient is divisible by $p$. This is easily determined from the following well-known result of Kummer:

\begin{prop}\cite[p.116]{Kummerwhatpowerpdivides}
\label{prop: Kummerbinomialcoef}
The highest power of a prime $p$ that divides the binomial coefficient $\binom{x+y}{x}$ is equal to the number of ``carries" that occur when the integers $x$ and $y$ are added in $p$-ary notation.
\end{prop}

Our first general example is the case $\lambda=(p^a, p^a)$ for $a \geq 1$. We first define the element $u \in M^{(p^a,p^a)}$ and then show (when $p>2$) that it satisfies (1) and (2) of Theorem \ref{thm: requirementonu}. For $0 \leq i \leq p^a-1$ define:
\begin{equation}
\label{eq: defofvi}
v_i=\sum \left\{\{t\} \in M^{(p^a,p^a)} \mid \text{ Exactly $i$ of $\{1,2,3, \ldots, p^a-1\}$ lie in row two of } \{t\} \right\}.
\end{equation}
Recall that for a two-part composition $\mu=(\mu_1, \mu_2)$ we are denoting $\mu$ tabloids in $M^\mu$ by the entries in the second row. For example
 
 $$\overline{1, 2, 3,  \cdots , t , p^a, p^a+1, \cdots,  p^a+s-t-1} $$ is an element of $M^{(2p^a-s, s)}$.
 
 It is straightforward to compute  $\psi_{1,s}(v_i)$:

\begin{lem}
\label{lem: psimapsonvi}
Let $v_i$ be as in \eqref{eq: defofvi}. Then:
\begin{itemize}
  \item [(a)] $\psi_{1,0}(v_i) = \binom{p^a-1}{i}\binom{p^a+1}{p^a-i} \overline{\emptyset}$.
  \item [(b)] For $ 1\leq t \leq s <p^a$, the coefficient of $$\overline{1, 2, 3,  \cdots , t , p^a, p^a+1, \cdots,  p^a+s-t-1} \in M^{(2p^a-s,s)}$$ in $\psi_{1,s}(v_i)$ is $\binom{p^a-1-t}{i-t}\binom{p^a-s+t+1}{p^a-s+t-i}.$
\end{itemize}
\end{lem}

\begin{proof}
Part (a) is just a count of the number of tabloids that appear in the sum defining $v_i$.  There are $\binom{p^a-1}{i}$ choices for the row two entries from $\{1, 2, \ldots, p^a-1\}$ and then $\binom{p^a+1}{p^a-i}$ for the remaining entries from $\{p^a, p^a+1, \ldots, 2p^a\}$.

For part (b) we count the $(p^a, p^a)$-tabloids in the sum defining $v_i$ that  contribute to that coefficient when plugged into $\psi_{1,s}$. Such a tabloid must have $\{1, 2, \ldots, t\}$ in the second row, so there are $\binom{p^a-1-t}{i-t}$ choices for the remaining $i-t$ entries from $\{1, 2, \ldots, p^a-1\}$ and then $\binom{p^a-s+t+1}{p^a-s+t-i}$ for the remaining entries from $\{p^a, p^a+1, \ldots, 2p^a\}$.
\end{proof}
\begin{rmk}
\label{rmk: coefssame}Lemma \ref{lem: psimapsonvi}(b) gives the coefficient in $\psi_{1,s}(v_i)$  of a tabloid containing in its second row precisely $\{1, 2, \ldots, t\}$ from among $\{1, 2, \ldots, p^a-1\}$ and $\{p^a+1, p^a+2, \ldots, , p^a+s-t-1\}$ from among $\{p^a+1, p^a+2, \ldots, 2p^a\}$.
However it is clear from the definition of $v_i$ that  any tabloid in $M^{(2p^a-s,s)}$ with second row containing exactly $t$ entries from $\{1,2, \ldots, p^a-1\}$ and $s-t$ entries from $\{p^a, p^a+1, \ldots, 2p^a\}$ will have  the same coefficient. Thus  Lemma \ref{lem: psimapsonvi}(b) gives a complete description of $\psi_{1,s}(v_i)$.
\end{rmk}

We now define the $u$ that, together with $S^{(p^a, p^a)}$, will give the nonsplit extension. Define:
\begin{equation}
\label{eq: defineuinpapacase}
u= \sum_{m=0}^{p^a-1}(m+1)v_m \in M^{(p^a, p^a)}.
\end{equation}
We will show $u$  satisfies condition (1) of Theorem \ref{thm: requirementonu} by first showing $\psi_{1,0}(u) \neq 0$ and then showing $\psi_{1,s}(u)=0$ for $s \geq 1$.

\begin{lem}
\label{lem: psi10uinpapacase}
For $u$ as in \eqref{eq: defineuinpapacase}, we have $\psi_{1,0}(u)=\overline{\emptyset}$
\end{lem}
\begin{proof}
By Prop \ref{prop: Kummerbinomialcoef} and Lemma \ref{lem: psimapsonvi}(a) we have $\psi_{1,0}(v_i)=0$ for all $i \not \in \{0, p^a-1\}$. Thus $$\psi_{1,0}(u)=\psi_{1,0}(v_0) +p^a\psi_{1,0}(v_{p^a-1})= \binom{p^a-1}{0}\binom{p^a+1}{p^a}\overline{\emptyset} = \overline{\emptyset}.$$
\end{proof}

Next we show $\psi_{1,i}(u)=0$ for all $i \geq 1.$
\begin{defin}
For $ 1\leq t \leq s <p^a$, let $A_{s,t}$ be the coefficient of
     \begin{equation}
            \label{eq: defofhugetabloid}
 \overline{1, 2, 3,  \cdots , t , p^a, p^a+1, \cdots,  p^a+s-t-1} \in M^{(2p^a-s, s)}
        \end{equation}
  in $\psi_{1,s}(u)$. The tabloid in \eqref{eq: defofhugetabloid} is just our canonical representative among the $(2p^a-s,s)$-tabloids with second row containing $t$ entries from $\{1, 2, \ldots, p^a-1\}$ and $s-t$ entries from $\{p^a, p^a+1, \ldots, 2p^a\}$. Knowing the coefficients of these tabloids gives all the coefficients, see Remark \ref{rmk: coefssame}.

From Lemma \ref{lem: psimapsonvi}(b) we have:
\begin{equation}
\label{eq: defofAst}
A_{s,t}=\sum_{m=t}^{p^a-1}(m+1)\binom{p^a-1-t}{m-t}\binom{p^a-s+t+1}{m+1}.
\end{equation}
\end{defin}
The next two lemmas combined will prove that all the $A_{s,t}, s \geq 1$ are divisible by $p$.
\begin{lem}
\label{lem: basecaseforass-1=0}
For $1 \leq s<p^a$, $A_{s,s-1} \equiv 0 \mod p.$
\end{lem}
\begin{proof}
When $t=s-1$ the second binomial coefficient in each term of \eqref{eq: defofAst} is $\binom{p^a}{m+1}$, which is congruent to zero except for the last term $m=p^a-1$, in which case the $(m+1)$ in front is zero.
\end{proof}
\begin{lem}
\label{lem: inductioncaseforass-1=0}
$1 \leq t \leq s<p^a$, we have:
\begin{equation}
A_{s,t}-A_{s,t-1}=\binom{2p^a-s-1}{p^a-1} \equiv 0 \mod  p.
\end{equation}
\end{lem}
\begin{proof} Apply the identity $$\binom{p^a-s+t+1}{m+1}=\binom{p^a-s+t}{m}+\binom{p^a-s+t}{m+1}$$ to the second binomial coefficient in \eqref{eq: defofAst}. Expand out and collect terms to obtain:
\begin{equation}
\begin{split}
\label{eq: astafterfirststep}
A_{s,t}=(t+1)\binom{p^a-t-1}{0}\binom{p^a-s+t}{t} + \\ \sum_{w=t}^{p^a-2}\left[(w+1)\binom{p^a-t-1}{w-t}+(w+2)\binom{p^a-t-1}{w-t+1}\right ]\binom{p^a-s+t}{w+1}+
\\p^a\binom{p^a-t-1}{p^a-t-1}\binom{p^a-s+t}{p^a}.
\end{split}
\end{equation}
Finally replace each $(w+1)\binom{p^a-t-1}{w-t}+(w+2)\binom{p^a-t-1}{w-t+1}$in \eqref{eq: astafterfirststep} by
$$(w+1)\binom{p^a-t}{w-t+1} + \binom{p^a-t-1}{w-t+1}.$$
and subtract off
\begin{equation}
\label{eq:Ast-1}
A_{s,t-1}=\sum_{w=t-1}^{p^a-1}(w+1)\binom{p^a-t}{w-t+1}\binom{p^a-s+t}{w+1}
\end{equation}
to obtain
\begin{eqnarray}
\label{eq: Ast-Ast-1}
A_{s,t}-A_{s,t-1}&=&\sum_{w=t}^{p^a-1}\binom{p^a-t-1}{w-t}\binom{p^a-s+t}{w}\\&=& \binom{2p^a-s-1}{p^a-1} \text{ by \cite[(5.23)]{GrahamKnuthPatashnikConcreteMathematics}}\nonumber\\& \equiv & 0 \text{ by Proposition \ref{prop: Kummerbinomialcoef}. }\nonumber
\end{eqnarray}
 The last congruence is clear from Proposition \ref{prop: Kummerbinomialcoef}. Expanding $p^a-1$ in $p$-ary notation, all the digits are $p-1$. Thus adding anything nonzero in $p$-ary notation will always result in at least one ``carry".
\end{proof}

\begin{thm}
\label{thm: pp} Let $k$ have characteristic $p \geq 3$. Then for any $a \geq 1$: $$\HH^1(\Sigma_{2p^a}, S^{(p^a,p^a)}) \neq 0.$$
\end{thm}

\begin{rmk} As mentioned earlier, we know this nonzero cohomology group is exactly one-dimensional, but this does not follow from our proof. However our proof gives an explicit basis for the nonsplit extension, which is new even in this case.
\end{rmk}

\begin{proof}We apply Theorem \ref{thm: requirementonu} to the $u$ defined in \eqref{eq: defineuinpapacase}. Lemmas \ref{lem: basecaseforass-1=0} and \ref{lem: inductioncaseforass-1=0} imply all the $A_{s,t}$ are congruent to 0 when $s \geq 1$, and thus $\psi_{1,s}(u)=0$ for $s \geq 1$. Together with Lemma \ref{lem: psi10uinpapacase}, we see $u$ satisfies part (1) of Theorem \ref{thm: requirementonu}. Finally note that:

$$\psi_{1,p^a-1}(f_{(p^a, p^a)}) = \binom{p^a+1}{1}f_{(2p^a-1,1)} \neq 0$$ but $\psi_{1,p^a-1}(u)=0$, so condition (2) is also satisfied.

\end{proof}
\begin{rmk}
Notice that the $u$ in \eqref{eq: ufor33} is \emph{not} the same as that in \eqref{eq: defineuinpapacase} for the case $p^a=3$, illustrating Remark \ref{rmk: manychoiceofu}. Of course the difference between the two lies in $S^{(3,3)}$.
\end{rmk}

Example \ref{exm: 83}  generalizes directly. Specifically we have:

\begin{thm} Let $\lambda=(p^b-1, p^a)$ for $a<b$ and let
\begin{equation}\nonumber
u=\sum \left\{\{t\} \in M^{\lambda} \mid 1,2,\ldots, p^a \text{ \rm appear in the first row of } \{t\} \right\}.
\end{equation}
Then $u$ satisfies Theorem \ref{thm: requirementonu} and thus $\HH^1(\Sigma_{p^b+p^a-1}, S^\lambda) \neq 0$.
\end{thm}
\begin{proof}
We leave the details to the reader. As in the $\lambda=(8,3)$ case all the $\psi_{1,s}$ vanish on $u$ except $\psi_{1,0}$. The verification is much more straightforward then the previous example, and does not require any identities involving binomial coefficients.
\end{proof}

\section{Further Directions}
\label{sec: further directions}
For a partition $\lambda=(\lambda_1, \lambda_2, \ldots) \vdash d$ define $p\lambda=(p\lambda_1, p\lambda_2, \ldots ) \vdash pd.$ In \cite{HemmerCohomologyandgenericSpecht} we proved the following:

\begin{thm}
\cite[Theorem 6.5.7]{HemmerCohomologyandgenericSpecht}
\label{thm: extinjectionontwistedSpecht} Let $\lambda \vdash d$ and let $p>2$. Then there is a isomorphism:
$$ \HH^1(\Sigma_{pd}, S^{p\lambda}) \cong \HH^1(\Sigma_{p^2d}, S^{p^2\lambda}) .$$

\end{thm}

The proof using a lot of algebraic group machinery, and does not produce an explicit map between the two cohomology groups; i.e. given a nonsplit extension of $S^{p\lambda}$ by $k$ as $\Sigma_{pd}$-modules, we can not obtain a corresponding $\Sigma_{p^2d}$-extension of $S^{p^2\lambda}$ by $k$. The following problem was the motivation for this paper:
\begin{prob}Suppose $\HH^1(\Sigma_{pd}, S^{p\lambda}) \neq 0$ and suppose one has constructed a $u \in M^{p\lambda}$ satisfying Theorem \ref{thm: requirementonu}. Describe a general method to construct a $\tilde{u} \in M^{p^2\lambda}$ corresponding to an element in $\HH^1(\Sigma_{p^2d}, S^{p^2\lambda})$ and realizing the isomorphism in Theorem \ref{thm: extinjectionontwistedSpecht}.
\end{prob}

\begin{prob} It is known \cite[Prop. 5.2.4]{HemmerCohomologyandgenericSpecht} that for $\lambda \neq (d)$, if $\HH^0(\Sigma_d, S^\lambda) \neq 0$ then $\HH^1(\Sigma_d, S^\lambda) \neq 0$. For each such $\lambda$ (given by Theorem  \ref{thm:JamestheoremonHom}) construct a $u$.
\end{prob}

The following is an easy consequence of Theorem \ref{thm:JamestheoremonHom}:

\begin{lem}
Suppose $\lambda=(\lambda_1, \lambda_2, \ldots, \lambda_s)\vdash d$ and suppose $a \equiv -1 \mod p^{l_p(\lambda_1)}$. Then
\begin{equation}
\label{eq:Homstabilizesaddingcong-1asnewfirstrow}
\HH^0(\Sigma_d, S^\lambda) \cong \HH^0(\Sigma_{d+a}, S^{(a,\lambda_1, \lambda_2, \ldots, \lambda_s)}).
\end{equation}
\end{lem}
This leads to the following

\begin{prob}
\label{prob:addingfirstrow-1stabilityonext}
Does the isomorphism in \eqref{eq:Homstabilizesaddingcong-1asnewfirstrow} hold for $\HH^i$ for any other $i>0$?
\end{prob}

Perhaps the $i=1$ version of Problem \ref{prob:addingfirstrow-1stabilityonext} can be attacked using Theorem \ref{thm: requirementonu}, i.e. given a $u$ that works for $\lambda$, produce one that works for $(a, \lambda_1, \lambda_2, \ldots, \lambda_s)$.

\bibliography{references0808}
\bibliographystyle{plain}
\end{document}